\documentclass[a4paper,12pt]{article}

\usepackage{amsmath,amssymb,amsfonts}
\usepackage{amsthm}
\usepackage{bm}
\usepackage{url}

\newtheorem{thm}{Theorem}
\newtheorem{lem}[thm]{Lemma}
\newtheorem{prop}[thm]{Proposition}
\newtheorem{cor}[thm]{Corollary}

\theoremstyle{definition}

\theoremstyle{remark}
\newtheorem{rem}[thm]{Remark}

\newcommand{\re}{\mathbb R}
\newcommand{\ze}{\mathbb Z}
\newcommand{\ex}{\mathbb E}
\newcommand{\pr}{\mathbb P}
\newcommand{\PD}{\operatorname{PD}}
\newcommand{\SL}{\operatorname{SL}}
\newcommand{\GH}{\operatorname{GH}}
\newcommand{\LI}{\operatorname{LI}}
\newcommand{\prim}{\operatorname{prim}}
\newcommand{\Var}{\operatorname{Var}}

\title{Thermal Concentration and Poisson--Dirichlet Edge Statistics for Random-Lattice Gibbs Ensembles}
\author{Masahiro Kaminaga\\
Tohoku Gakuin University, Sendai, Japan\\
E-mail: masahiro.kaminaga@gmail.com, \\
kaminaga@mail.tohoku-gakuin.ac.jp}
\date{}

\begin{document}
\maketitle

\begin{abstract}
We study Gibbs measures on high-dimensional Haar-random unimodular lattices, where the energy of a lattice vector is its squared Euclidean norm.
The random lattice is viewed as quenched geometric disorder, and $c>0$ denotes the scaled inverse temperature.
We first analyze the edge window of vectors whose length is within the factor $e^{a/n}$ of the shortest length, with fixed $a$ as $n\to\infty$.
For the full sign-class Gibbs ensemble, we prove a Poisson point process limit theorem for the Gibbs mass of this window.
The mass vanishes in probability for $0<c\le1$, while for $c>1$ it has a nontrivial limit given by a Poisson--process functional, and the ranked Gibbs weights converge to the Poisson--Dirichlet distribution with parameter $1/c$.
We then pass to a primitive-direction Gibbs ensemble and consider a fixed approximation factor $\gamma>1$.
For this modified ensemble, we prove a weighted moment formula and a quenched thermal concentration result in the high-temperature range $0<c<1$.
This yields the primitive fixed-factor visibility curve $c=\gamma^{-2}$ for approximate shortest directions.
More precisely, the primitive Gibbs mass of the fixed-factor window tends to zero for $c<\gamma^{-2}$, to one for $\gamma^{-2}<c<1$, and to $1/2$ at the critical boundary $c=\gamma^{-2}$.
Thus the fixed-factor theorem is a visibility statement for an idealized primitive target measure, not for the original full lattice Gibbs measure.
The results provide a random-lattice thermodynamic reference model for Gibbs targets related to approximate shortest vectors, without implying an efficient algorithm for the shortest vector problem.
\end{abstract}


\section{Introduction}
\label{sec:introduction}
The purpose of this paper is to study a simple Gibbs ensemble on
high-dimensional random lattices.
Let $L\subset\re^n$ be a unimodular lattice.
We put
$$
E(\bm v)=\|\bm v\|^2,
\qquad \bm v\in L\setminus\{\bm 0\}.
$$
The Gibbs distribution at inverse temperature $\beta$ is proportional to
$\exp(-\beta\|\bm v\|^2)$.
Thus, after excluding the zero vector, the shortest nonzero vectors play the role of ground states.
We denote the length of a shortest nonzero vector by
$$
\lambda_1(L)=
\min\left\{\|\bm v\|:\bm v\in L\setminus\left\{\bm 0\right\}\right\}.
$$
Vectors whose length is close to $\lambda_1(L)$ form a low-energy edge.
The question considered here is when this edge, or a fixed approximation
window around it, carries visible Gibbs mass.
From a statistical--mechanical point of view, this is a visibility problem for
an equilibrium target measure.
The state space is random and quenched, while the energy is deterministic.
The main issue is whether the entropy coming from many nearly low-energy
states can overcome their energy penalty.

Gaussian weights on Euclidean lattices are classical objects.
They appear in theta series, Gaussian lattice sums, Epstein zeta functions,
and discrete Gaussian distributions in lattice cryptography.
The present paper studies normalized Gibbs probabilities of shortest-vector
approximation windows, rather than unnormalized lattice sums.
The lattice is the quenched geometric disorder and the set of lattice points is the random state space.
Thus the model is a simple disordered geometric Gibbs ensemble in which a
deterministic quadratic energy competes with the growth of available states.
It is similar in spirit to random energy models~\cite{Derrida80,Derrida81},
but the energies here are constrained by lattice geometry and are not independent.
The model also gives a random-lattice comparison setting for questions
related to approximate shortest vectors.
In this setting, the edge near the shortest length and the
entropy--energy balance determining the visibility of approximation
windows can both be analyzed explicitly.

The main random model is the space
$$
\mathbb X_n=\SL(n,\ze)\backslash\SL(n,\re)
$$
of unimodular lattices in $\re^n$ with its invariant probability measure $\mu_n$.
Here $\SL(n,\re)$ denotes the group of real $n\times n$ matrices with
determinant one, and $\SL(n,\ze)$ is its subgroup consisting of integer matrices.
This quotient is identified with unimodular lattices in $\re^n$.
We write $\re_+=(0,\infty)$.
For a nonzero vector modulo sign, put
$$
\tau_n(\bm u)=v_n\|\bm u\|^n,
$$
where $v_n$ is the volume of the Euclidean unit ball in $\re^n$.
S\"odergren's theorem says that the point process formed by these volume
variables converges to a Poisson process on $\re_+$ with intensity $1/2$ as $n\to\infty$ \cite{Sodergren11}.
The precise meaning of this point--process convergence will be recalled in Section~\ref{sec:random}.
This Poisson limit is the probabilistic basis for the edge part of the paper.
The temperature scale used here is
$$
\beta_n(c)=\frac{cn}{2}v_n^{2/n}, \qquad c>0.
$$
The parameter $c$ is the scaled inverse temperature in this normalization.
In this paper, the high-temperature range means the scaled range $0<c<1$.
In the edge theorem, the range $c>1$ is the corresponding low--temperature edge--condensed range.
Since
$$
v_n^{2/n}\|\bm u\|^2=\tau_n(\bm u)^{2/n},
$$
a common factor may be removed from the Gibbs probabilities.
The relevant weight is
$$
g_{n,c}(s)=\exp\left\{-\frac{cn}{2}(s^{2/n}-1)\right\}.
$$
For fixed $s>0$, this weight converges to $s^{-c}$.
Thus the low-energy Gibbs weights are governed by power weights on 
the limiting Poisson cloud, and the summability of these power weights changes at $c=1$.
The thresholds obtained below are visibility thresholds of the target Gibbs measure.

The first main result is an edge theorem for the full sign-class Gibbs ensemble.
For the approximation factor $\gamma_n=e^{a/n}$, the event
$\|\bm u\|\le \gamma_n\lambda_1(L)$ is the same as $\tau_n(\bm u)\le e^a\tau_{n,1}(L)$.
We prove that its Gibbs mass tends to zero in probability for $0<c\le1$.
For $c>1$, it converges in distribution to an explicit functional of the limiting Poisson process.
In the same regime, the full ranked Gibbs mass partition converges to the
Poisson--Dirichlet distribution $\PD(1/c,0)$.
This is a condensation statement at the random-lattice edge:
in the low--temperature phase, the Gibbs measure does not spread evenly over
many short vectors, and a random ranked mass partition survives in the limit.

The second main result concerns fixed approximation factors $\gamma>1$.
For such a window, the corresponding interval in the volume variable is
exponentially large, and the Poisson edge theorem is no longer sufficient.
We therefore pass to the exponential coordinate
$$
Y_n(\bm u)=\frac{1}{n}\log\tau_n(\bm u).
$$
The term $y$ represents the exponential growth of the number of available
lattice points on the volume scale $\tau_n\simeq e^{ny}$, while the term
$-(c/2)(e^{2y}-1)$ comes from the Gibbs weight.
The entropy and energy contributions combine into
$$
\Phi_c(y)=y-\frac{c}{2}(e^{2y}-1),\qquad y\ge0.
$$
In the high-temperature range $0<c<1$, this function has its unique maximum at
$$
y_c=\frac{1}{2}\log\frac{1}{c}.
$$
The fixed-factor theorem is not stated for the original full lattice Gibbs ensemble.
Instead, it is stated for a primitive-direction Gibbs ensemble.
This restriction is natural for the study of shortest-vector directions,
because every nonprimitive vector is an integer multiple of a shorter
vector in the same lattice direction.
For these primitive sign classes, we prove a weighted moment formula and use it
to show that the primitive Gibbs mass concentrates, in probability, near $y_c$.
It follows that the primitive Gibbs mass of a fixed approximation window has
the sharp visibility curve
$$
c=\gamma^{-2}.
$$
This curve is an entropy--energy balance.
Below it, the approximation window misses the typical thermal scale.
Above it, the same window contains the typical thermal scale.
At the critical boundary $c=\gamma^{-2}$, a local Laplace estimate shows that
the primitive Gibbs mass of the fixed approximation window converges to $1/2$.
Thus the fixed-factor result is a quenched high-dimensional concentration
theorem for an idealized primitive target measure, rather than a statement
about the original full lattice Gibbs measure.
To the best of our knowledge, this explicit primitive fixed-factor visibility
curve has not been formulated in this random-lattice Gibbs setting.
The main regimes are summarized in Table~\ref{tab:main-regimes}.

\begin{table}[t]
\centering
\footnotesize
\setlength{\tabcolsep}{3pt}
\renewcommand{\arraystretch}{1.15}
\begin{tabular}{|p{0.20\textwidth}|p{0.28\textwidth}|p{0.22\textwidth}|p{0.21\textwidth}|}
\hline
Regime & Window and ensemble & Condition on $c$ & Limit \\
\hline
Edge, invisible &
$e^{a/n}$; full sign classes &
$0<c\leq 1$ &
mass tends to $0$ \\
\hline
Edge, condensed &
$e^{a/n}$; full sign classes &
$c>1$ &
Poisson functional; $\PD(1/c,0)$ partition \\
\hline\hline
Fixed factor, noncritical &
$\gamma>1$; primitive sign classes &
$0<c<\gamma^{-2}$ or $\gamma^{-2}<c<1$ &
mass tends to $0$ or $1$, respectively \\
\hline
Critical fixed factor &
$\gamma>1$; primitive sign classes &
$c=\gamma^{-2}$ &
mass tends to $1/2$ \\
\hline
\end{tabular}
\caption{Summary of the visibility regimes considered in this paper.}
\label{tab:main-regimes}
\end{table}
The present work is related to several existing uses of Gaussian and power weights on lattices.
Theta series and Gaussian lattice sums give partition functions of Euclidean
lattices, while Epstein zeta functions give power--weighted analogues.
In high dimension, S\"odergren studied the value distribution and moments of
$E_n(L,s)$ to the right of the critical strip and also the value distribution
in the critical strip \cite{Sodergren10,Sodergren13}.
Those works concern unnormalized sums of weights over lattice vectors.
The present paper instead studies normalized Gibbs masses, ranked mass
partitions, and thermal visibility of approximation windows.

The shortest vector problem (SVP) gives one motivation for this study.
For a full-rank lattice, SVP asks for a nonzero vector of minimum Euclidean norm; its approximate version asks for a vector within a prescribed factor of this minimum.
Such questions are related to lattice--based cryptography, where worst--case to average--case reductions for the short integer solution (SIS) problem and the learning with errors (LWE) problem play a central role \cite{Ajtai96,MG02,Regev09,MicciancioRegev07}.
The cryptographic role of the present result is only indirect: it provides a random-lattice comparison setting for Gibbs or thermal targets associated with approximate SVP, but it does not address the security of concrete structured lattice families.
The results are static visibility statements for idealized target Gibbs
measures and do not imply that an efficient sampling algorithm exists.
Nevertheless, such visibility statements are useful as a thermodynamic reference point.
If a target Gibbs measure puts negligible mass on a desired approximation
window, then an ideal sample from that target cannot find such a vector with
nonnegligible probability.

%%%%%%%%%%%%%%%%%%%%%%%%%%%%%%%%%%%%%%%%%%%%%%%%
\section{Lattice Gibbs Ensembles}
\label{sec:gibbs}

Let $L\subset\re^n$ be a full-rank lattice.
We write
$$
\lambda_1(L)=\min\{\|\bm v\|:\bm v\in L\setminus\{\bm 0\}\}.
$$
For $\gamma\ge1$, the $\gamma$--approximate shortest vector problem asks for a nonzero $\bm v\in L$ satisfying
$$
\|\bm v\|\le\gamma\lambda_1(L).
$$
We identify opposite vectors and write
$$
L^\sharp=(L\setminus\{\bm 0\})/\{\pm1\}
$$
for the set of sign classes of nonzero lattice vectors.
For $\bm u\in L^\sharp$, the symbol $\|\bm u\|$ means the norm of either representative.

For $\beta>0$, define
$$
Z_L^\sharp(\beta)=\sum_{\bm u\in L^\sharp}e^{-\beta\|\bm u\|^2}.
$$
The Gibbs probability of a sign class $\bm u$ is
$$
\pi_{L,\beta}(\bm u)=
\frac{e^{-\beta\|\bm u\|^2}}{Z_L^\sharp(\beta)}.
$$
The Gibbs probability mass assigned to the $\gamma$--approximate set is
$$
P_{L,\beta}(\gamma)=
\sum_{\bm u\in L^\sharp,\ \|\bm u\|\le\gamma\lambda_1(L)}
\pi_{L,\beta}(\bm u).
$$
This number is the success probability of one ideal Gibbs sample.

For a fixed lattice, the zero-temperature limit is simple.
Let
$$
\kappa^\sharp(L)
=
\#\{\bm u\in L^\sharp:\|\bm u\|=\lambda_1(L)\}.
$$
Thus $\kappa^\sharp(L)$ is the number of shortest sign classes, or equivalently one half of the number of shortest nonzero vectors in $L$.
Since the length spectrum of a fixed full-rank lattice is discrete, the non--shortest sign classes are separated from the shortest ones by a positive energy gap.
Hence
$$
Z_L^\sharp(\beta)
=
\kappa^\sharp(L)e^{-\beta\lambda_1(L)^2}
\{1+o(1)\}
$$
as $\beta\to\infty$.
Indeed, the shortest sign classes give the leading contribution, while every other sign class has strictly larger energy.
The following observation explains why approximation windows, rather than exact
shortest vectors alone, are natural at finite temperature.

\begin{rem}
For a fixed lattice, the zero-temperature limit is concentrated on the
shortest sign classes.
At finite temperature, however, an approximation window may carry visible mass
when the number of near--shortest sign classes compensates for their energy
penalty.
This is one motivation for studying Gibbs mass of approximation windows rather
than only exact shortest vectors.
\end{rem}

\section{Random Lattices and Edge Statistics}
\label{sec:random}

Let $\mu_n$ be the invariant probability measure on
$$
\mathbb X_n=\SL(n,\ze)\backslash\SL(n,\re).
$$
For $L\in\mathbb X_n$ and $\bm u\in L^\sharp$, put
$$
\tau_n(\bm u)=v_n\|\bm u\|^n.
$$
Let
$$
\tau_{n,1}(L)=\min_{\bm u\in L^\sharp}\tau_n(\bm u).
$$
We write $\Rightarrow$ for convergence in distribution of random variables,
or more generally of random elements.
We use the following theorem of S\"odergren \cite{Sodergren11}.

\begin{thm}[S\"odergren's Poisson limit theorem]
\label{thm:sodergren}
Let $L_n$ be distributed according to $\mu_n$ on $\mathbb X_n$.
The point process
$$
\xi_n=\sum_{\bm u\in L_n^\sharp}\delta_{\tau_n(\bm u)}
$$
converges in distribution to a Poisson point process $\xi$ on $\re_+$ with intensity $1/2$.
Here $\delta_t$ denotes the Dirac measure at $t$.
If $0<T_1<T_2<\cdots$ are the points of $\xi$, then the finite--dimensional statistics of the short sign classes converge to the corresponding statistics of these points.
\end{thm}

For $c>0$ define
$$
\beta_n(c)=\frac{cn}{2}v_n^{2/n}.
$$
Since $v_n^{2/n}\|\bm u\|^2=\tau_n(\bm u)^{2/n}$, the common factor $e^{-cn/2}$ may be removed without changing Gibbs probabilities.
We therefore put
$$
g_{n,c}(s)=\exp\left\{-\frac{cn}{2}(s^{2/n}-1)\right\}
$$
and
$$
M_{n,c}(L)=\sum_{\bm u\in L^\sharp}g_{n,c}(\tau_n(\bm u)).
$$
The Gibbs probability of a sign class is
$$
\pi_{n,c,L}(\bm u)=
\frac{g_{n,c}(\tau_n(\bm u))}{M_{n,c}(L)}.
$$
For $a\ge0$, define
$$
P_{n,c}(a,L)=
\sum_{\bm u\in L^\sharp,\ \tau_n(\bm u)\le e^a\tau_{n,1}(L)}
\pi_{n,c,L}(\bm u).
$$
This is the Gibbs mass of the sign classes satisfying
$$
\|\bm u\|\le e^{a/n}\lambda_1(L).
$$

\begin{lem}
\label{lem:powersum}
Let $0<T_1<T_2<\cdots$ be the points of a Poisson process on $\re_+$ with intensity $1/2$.
Then
$$
\sum_{j=1}^{\infty}T_j^{-c}<\infty
$$
almost surely if $c>1$, while the same sum diverges almost surely if $0<c\le1$.
\end{lem}

\begin{proof}
For a Poisson process with intensity $1/2$, the counting function $N(R)=\#\{j:T_j\le R\}$ satisfies $N(R)/R\to1/2$ almost surely.
Equivalently, $T_j/j\to2$ almost surely.
Hence $T_j^{-c}$ has the same summability behavior as $(2j)^{-c}$.
\end{proof}

\begin{lem}
\label{lem:tail}
If $c>1$ and $R\ge1$, then for every $n$,
$$
\ex_{\mu_n}\left[\sum_{\bm u\in L^\sharp,\ \tau_n(\bm u)>R}
 g_{n,c}(\tau_n(\bm u))\right]
\le
\frac{R^{1-c}}{2(c-1)}.
$$
\end{lem}

\begin{proof}
For $s\ge1$, the inequality $e^x\ge1+x$ with $x=(2/n)\log s$ gives $s^{2/n}-1\ge(2/n)\log s$.
Hence $g_{n,c}(s)\le s^{-c}$ for $s\ge1$.
Siegel's mean value formula \cite{Siegel45} in the volume variable gives
$$
\ex_{\mu_n}\left[\sum_{\bm u\in L^\sharp,\ \tau_n(\bm u)>R}
 g_{n,c}(\tau_n(\bm u))\right]
=
\frac{1}{2}\int_R^\infty g_{n,c}(s)\,ds.
$$
The last integral is at most $\int_R^\infty s^{-c}\,ds$.
This proves the estimate.
\end{proof}

\section{Poisson--Dirichlet Edge Condensation}
\label{sec:edge}

The first main theorem describes the edge approximation factor $e^{a/n}$.
This regime is very close to the exact shortest vector, but it is the regime in which the short--vector Poisson process gives a complete limit theorem.

\begin{thm}[Edge approximate visibility]
\label{thm:edge}
Let $L_n$ be distributed according to $\mu_n$ on $\mathbb X_n$.
Fix $a\ge0$.
If $0<c\le1$, then
$$
P_{n,c}(a,L_n)\to0
$$
in probability as $n\to\infty$.
If $c>1$, then
$$
P_{n,c}(a,L_n)
\Rightarrow
Q_c(a):=
\frac{\sum_{T_j\le e^aT_1}T_j^{-c}}
{\sum_{j=1}^{\infty}T_j^{-c}},
$$
where $0<T_1<T_2<\cdots$ are the points of a Poisson process on $\re_+$ with intensity $1/2$.
\end{thm}

\begin{proof}
First assume $0<c\le1$.
We denote the numerator of $P_{n,c}(a,L_n)$ by
$$
A_{n,c}(a,L_n)=
\sum_{\bm u\in L_n^\sharp,\ \tau_n(\bm u)\le e^a\tau_{n,1}(L_n)}
g_{n,c}(\tau_n(\bm u)).
$$
This is the numerator of $P_{n,c}(a,L_n)$.
We first show that $A_{n,c}(a,L_n)$ is tight.
Let $\epsilon>0$.
Since $\tau_{n,1}(L_n)\Rightarrow T_1$ and $0<T_1<\infty$ almost surely, we may choose
$0<r<R<\infty$ such that
$$
\liminf_{n\to\infty}
\pr\left\{r\le \tau_{n,1}(L_n)\le R\right\}\ge 1-\epsilon.
$$
On this event, the random window
$[\tau_{n,1}(L_n),e^a\tau_{n,1}(L_n)]$ is contained in $[r,e^aR]$.
By Theorem \ref{thm:sodergren}, the number of points of $\xi_n$ in
$[r,e^aR]$ is tight.
Moreover, $g_{n,c}$ converges uniformly to $s^{-c}$ on $[r,e^aR]$.
Hence there is a constant $C=C(r,R,a,c)$ such that
$$
\sup_{s\in[r,e^aR]}g_{n,c}(s)\le C
$$
for all sufficiently large $n$.
Therefore, on the above event,
$$
A_{n,c}(a,L_n)
\le
C\xi_n([r,e^aR])
$$
for all sufficiently large $n$.
Since the right hand side is tight and the exceptional event has limiting
probability at most $\epsilon$, the numerator $A_{n,c}(a,L_n)$ is tight.

The denominator diverges in probability.
For fixed $R>1$, put
$$
M_{n,c}^{[1,R]}(L_n)=
\sum_{\bm u\in L_n^\sharp,\ 1\le\tau_n(\bm u)\le R}
 g_{n,c}(\tau_n(\bm u)).
$$
On $[1,R]$, the functions $g_{n,c}$ converge uniformly to $s^{-c}$.
Hence Theorem \ref{thm:sodergren} gives
$$
M_{n,c}^{[1,R]}(L_n)
\Rightarrow
\sum_{1\le T_j\le R}T_j^{-c}.
$$
By Lemma \ref{lem:powersum}, the last sum tends to infinity almost surely as $R\to\infty$ when $0<c\le1$.
More explicitly, for every $K>0$ and $\epsilon>0$, one can first choose $R$ so that the limiting truncated sum exceeds $K$ with probability at least $1-\epsilon$, and then use the preceding convergence in distribution.
Since $M_{n,c}(L_n)\ge M_{n,c}^{[1,R]}(L_n)$, it follows that $M_{n,c}(L_n)$ tends to infinity in probability.
The ratio therefore tends to zero in probability.

Now assume $c>1$.
Let $0<r<R$.
On $[r,R]$, the functions $g_{n,c}$ converge uniformly to $s^{-c}$.
The point process convergence gives convergence of the finite weighted point configurations in this interval.
The lower cutoff is removed because $\tau_{n,1}(L_n)\Rightarrow T_1$ and $\pr(T_1\le r)\to0$ as $r\downarrow0$.
The upper cutoff is removed by Lemma \ref{lem:tail} and Markov's inequality, while Lemma \ref{lem:powersum} controls the corresponding tail of the limiting process.
Thus the numerator and denominator of $P_{n,c}(a,L_n)$ converge jointly to
$$
\sum_{T_j\le e^aT_1}T_j^{-c}
\qquad {\rm and}\qquad
\sum_{j=1}^{\infty}T_j^{-c}.
$$
The limiting denominator is positive and finite almost surely.

For a locally finite point configuration
$\xi=\sum_{j\geq 1}\delta_{T_j(\xi)}$ on $\re_+$, we write
$0<T_1(\xi)<T_2(\xi)<\cdots$ for its ordered points.
For $a>0$, the functional
$$
\xi\mapsto
\sum_{T_j(\xi)\leq e^a T_1(\xi)}
T_j(\xi)^{-c}
$$
is continuous at the limiting Poisson configuration almost surely.
Indeed, the limiting point process is simple, $T_1>0$ almost surely,
and there is almost surely no point at the random boundary $e^aT_1$.
In the present regime $c>1$, the denominator functional
$$
\xi\mapsto \sum_{j=1}^{\infty}T_j(\xi)^{-c}
$$
is also almost surely continuous, because the limiting sum is finite and
the tail contribution is summable.
Hence the continuous mapping theorem gives the stated convergence for
$a>0$.
When $a=0$, the numerator is just $T_1(\xi)^{-c}$.
Since the first point is isolated almost surely, this functional is also
continuous at the limiting Poisson configuration.
The same denominator argument applies, and the conclusion follows also
for $a=0$.
\end{proof}

\begin{thm}[Poisson--Dirichlet mass partition]
\label{thm:pd}
Let $L_n$ be distributed according to $\mu_n$ on $\mathbb X_n$.
Let $c>1$ and put $\alpha=1/c$.
Order the sign classes of $L_n^\sharp$ by increasing $\tau_n$ and write them as $\bm u_{n,1},\bm u_{n,2},\ldots$.
Define
$$
W_{n,j}^{(c)}=
\frac{g_{n,c}(\tau_n(\bm u_{n,j}))}
{\sum_{\bm u\in L_n^\sharp}g_{n,c}(\tau_n(\bm u))}
$$
and $W_{n,c}=(W_{n,1}^{(c)},W_{n,2}^{(c)},\ldots)$.
Then, as a random element of $\ell^1$,
$$
W_{n,c}\Rightarrow W_c:=
\left(
\frac{T_1^{-c}}{\sum_{k=1}^{\infty}T_k^{-c}},
\frac{T_2^{-c}}{\sum_{k=1}^{\infty}T_k^{-c}},
\ldots
\right).
$$
Since $0<T_1<T_2<\cdots$, the sequence
$T_1^{-c},T_2^{-c},\ldots$ is already ordered in decreasing order.
Thus $W_c$ is a normalized ranked mass sequence.
Moreover, $W_c$ has the Poisson--Dirichlet distribution $\PD(1/c,0)$.
\end{thm}

\begin{proof}
We first prove convergence of the unnormalized weight sequences.
For $0<r<R$, let $V_{n,c}^{r,R}$ be the sequence of weights
$g_{n,c}(\tau_n(\bm u))$ with $r\le \tau_n(\bm u)\le R$, ordered by increasing $\tau_n$ and padded with zeros.
Define $V_c^{r,R}$ analogously from the weights $T_j^{-c}$ with $r\le T_j\le R$.
On $[r,R]$, the functions $g_{n,c}$ converge uniformly to $s^{-c}$, and the point--process convergence therefore gives
$$
V_{n,c}^{r,R}\Rightarrow V_c^{r,R}
$$
in $\ell^1$.
The lower cutoff is negligible because
$$
\lim_{r\downarrow0}\limsup_{n\to\infty}
\pr\{\tau_{n,1}(L_n)<r\}=0.
$$
For the upper cutoff, Lemma \ref{lem:tail} and Markov's inequality give, for every $\epsilon>0$,
$$
\sup_n\pr\left\{
\sum_{\bm u\in L_n^\sharp,\ \tau_n(\bm u)>R}
 g_{n,c}(\tau_n(\bm u))>\epsilon
\right\}
\le
\frac{R^{1-c}}{2(c-1)\epsilon},
$$
which tends to zero as $R\to\infty$.
Lemma \ref{lem:powersum} gives the corresponding tail control for the limiting sequence.
A converging--together argument now yields
$$
\bigl(g_{n,c}(\tau_n(\bm u_{n,1})),
      g_{n,c}(\tau_n(\bm u_{n,2})),\ldots\bigr)
\Rightarrow
(T_1^{-c},T_2^{-c},\ldots)
$$
in $\ell^1$.
Since the limiting total mass is finite and positive almost surely, normalization is continuous on this event and gives the asserted convergence of $W_{n,c}$.

It remains to identify the law of the limit.
The image process
$$
\eta_c=\sum_{j=1}^{\infty}\delta_{T_j^{-c}}
$$
on $(0,\infty)$ is a Poisson point process with intensity measure
$$
\nu_c(dy)=\frac{1}{2c}y^{-1-1/c}\,dy.
$$
This is, up to multiplication by a positive constant, the jump measure of a stable subordinator of index $\alpha=1/c$.
The normalized ranked jumps of such a stable subordinator have the Poisson--Dirichlet distribution $\PD(\alpha,0)$; see \cite{Kingman75}; see especially \cite{PitmanYor97}.
Hence $W_c$ has distribution $\PD(1/c,0)$.
\end{proof}

\begin{prop}
\label{prop:cost}
Let
$$
P_c=\frac{T_1^{-c}}{\sum_{j=1}^{\infty}T_j^{-c}},
\qquad c>1.
$$
Then
$$
\ex P_c^{-1}=\frac{c}{c-1}.
$$
\end{prop}

\begin{proof}
We have
$$
P_c^{-1}=1+\sum_{j=2}^{\infty}\left(\frac{T_1}{T_j}\right)^c.
$$
Conditioned on $T_1=t$, the points larger than $t$ form a Poisson process on $(t,\infty)$ with intensity $1/2$.
Hence
$$
\ex\left[
\sum_{j=2}^{\infty}\left(\frac{t}{T_j}\right)^c
\middle|T_1=t\right]
=
\frac{1}{2}\int_t^\infty\left(\frac{t}{x}\right)^c dx
=
\frac{t}{2(c-1)}.
$$
The first point $T_1$ is exponential with rate $1/2$, so $\ex T_1=2$.
This gives $\ex P_c^{-1}=1+1/(c-1)=c/(c-1)$.
\end{proof}

\section{Primitive Weighted Moment Formula}
\label{sec:primitive}

The fixed-factor approximation window contains exponentially many lattice points on the volume scale.
For this reason one needs more than the edge Poisson theorem.
We use primitive sign classes in order to remove the trivial copies $m\bm v$ along the same one--dimensional lattice subgroup.
This is natural for shortest-vector questions, because a shortest vector is primitive and nonprimitive copies do not create new directions.

We use the notation
$$
\GH_n=v_n^{-1/n}.
$$
Thus the ball of radius $\GH_n$ has volume one.
A nonzero vector $\bm v\in L$ is called primitive if there are no integer $m\ge2$ and vector $\bm w\in L$ such that $\bm v=m\bm w$.
Let $L_{\prim}$ be the set of primitive nonzero vectors in $L$, and let $L_{\prim}^\sharp$ be the set of primitive sign classes in $L$.
We also write $\ze^n_{\prim}$ for the set of primitive nonzero vectors in $\ze^n$.
For a set $D$, the notation $\chi_D$ denotes its indicator function.

We use the row--vector convention in this section.
Thus $L=\ze^n A$ with $A\in\SL(n,\re)$.
Let $\LI_k$ be the set of ordered $k$--tuples $\bm q=(\bm q_1,\ldots,\bm q_k)$ in $(\ze^n)^k$ which are linearly independent over $\re$.
For integer vectors, linear independence over $\re$ is equivalent to linear
independence over $\mathbb Q$.
Indeed, real linear dependence is equivalent to the vanishing of all
$k\times k$ minors, and then the rank over $\mathbb Q$ is also less than $k$.
We shall use the formula below only for $k=1$ and $k=2$.
Thus, in the second moment argument, the condition $1\le k\le n-1$ is satisfied
under the assumption $n\ge3$.
We use the following form of the Rogers--Macbeath--Rogers mean value formula
\cite{Rogers55Mean,Rogers55Moments,MacbeathRogersI,MacbeathRogersII}.

\begin{thm}[Macbeath--Rogers]
\label{thm:macbeath}
Let $1\le k\le n-1$.
For every nonnegative Borel function $f$ on $(\re^n)^k$,
$$
\int_{\mathbb X_n}\sum_{\bm q\in\LI_k}f(\bm q_1A,\ldots,\bm q_kA)\,d\mu_n(L)
=
\int_{(\re^n)^k}f(\bm x_1,\ldots,\bm x_k)\,d\bm x_1\cdots d\bm x_k,
$$
where $A$ is any representative of $L=\ze^nA$.
\end{thm}
We use Theorem~\ref{thm:macbeath} only for $k=1,2$.
For $k=2$, we use precisely the Macbeath--Rogers formula for ordered
linearly independent integer pairs.
Thus the formula gives only the rank--two, off--diagonal contribution to the second moment.
The rank--one dependent contributions are not discarded; in the
primitive sign-class sum they reduce to the diagonal term.
Indeed, two collinear primitive integer vectors are equal up to sign.
Hence, on primitive sign classes, the only collinear contribution is
$u=w$, and this is accounted for by applying the first moment formula to $f^2$.
\begin{lem}
\label{lem:mobius-primitive}
Let $\mu:\mathbb N\to\{-1,0,1\}$ be the M\"obius function. 
Thus $\mu(1)=1$, $\mu(d)=0$ if $d$ is divisible by the square of a prime, and
$\mu(d)=(-1)^r$ if
$$
d=p_1p_2\cdots p_r
$$
with distinct primes $p_1,\ldots,p_r$.

The elementary M\"obius inversion identity is
$$
\sum_{d|N}\mu(d)=
\begin{cases}
1, & N=1,\\
0, & N\ge2.
\end{cases}
$$
Consequently, for every nonzero vector $\bm q\in\ze^n$,
$$
\chi_{\ze^n_{\prim}}(\bm q)=\sum_{d|\bm q}\mu(d),
$$
where $d|\bm q$ means that $d$ divides every coordinate of $\bm q$.
\end{lem}

\begin{proof}
Let $N\ge1$. If $N=1$, the first identity is clear. If $N\ge2$, let
$p_1,\ldots,p_s$ be the distinct prime divisors of $N$. 
Since $\mu(d)=0$ unless $d$ is square--free, we have
$$
\sum_{d|N}\mu(d)
=
\sum_{J\subset\{1,\ldots,s\}}(-1)^{|J|}
=
(1-1)^s
=
0.
$$
Now let $\bm q=(q_1,\ldots,q_n)\ne0$ and put $g=\gcd(q_1,\ldots,q_n)$. 
Then the condition $d|\bm q$ is equivalent to $d|g$. 
Hence
$$
\sum_{d|\bm q}\mu(d)=\sum_{d|g}\mu(d).
$$
By the preceding identity, this is equal to $1$ if $g=1$ and to $0$ otherwise.
This is exactly the assertion that $\bm q$ is primitive.
\end{proof}
The next proposition uses the case $k=2$ of Theorem \ref{thm:macbeath} in the
second moment computation, and this is why we assume $n\ge3$.
\begin{prop}
\label{prop:weighted}
Assume $n\ge3$.
Let $f:\re^n\to[0,\infty)$ be an even Borel function such that $f$ and $f^2$ are integrable.
For $L\in\mathbb X_n$, define
$$
S_f(L)=\sum_{\bm u\in L_{\prim}^\sharp} f(\bm u),
$$
where $f(\bm u)$ means $f(\bm v)$ for either representative of $\bm u$.
Then, with respect to the random choice of $L$ according to $\mu_n$, the
mean and variance of $S_f(L)$ are given by 
$$
\ex_{\mu_n}S_f(L)=
\frac{1}{2\zeta(n)}\int_{\re^n}f(\bm x)\,d\bm x
$$
and
$$
\Var_{\mu_n}(S_f(L))=
\frac{1}{2\zeta(n)}\int_{\re^n}f(\bm x)^2\,d\bm x.
$$
\end{prop}

\begin{proof}
We first prove the identities for bounded compactly supported $f$.
In this case all lattice sums are finite for each $L$.
The sums over $d$ and over $d,e$ are absolutely summable after taking expectation, since
$$
\sum_{d=1}^{\infty}d^{-n}<\infty
$$
for $n\ge3$.
Thus the changes of order of summation, expectation, and integration below are justified.
For brevity, let
$$
I_f=\int_{\re^n}f(\bm x)\,d\bm x.
$$
We use Lemma \ref{lem:mobius-primitive}.
Since primitivity is preserved by multiplication by $A\in\SL(n,\re)$, Siegel's formula \cite{Siegel45} gives
\begin{eqnarray*}
\ex_{\mu_n}\sum_{\bm v\in L_{\prim}}f(\bm v)
&=&
\sum_{d=1}^{\infty}\mu(d)
\ex_{\mu_n}
\sum_{\bm q\in\ze^n\setminus\{\bm 0\}}
f(d\bm q A)
\\
&=&
\sum_{d=1}^{\infty}\mu(d)d^{-n} I_f
\\
&=&
\frac{1}{\zeta(n)}I_f.
\end{eqnarray*}
Since $f$ is even, passing from primitive vectors to primitive sign classes divides this by $2$.
This proves the first formula.
For the second moment, we write
$$
S_f(L)^2
=
\sum_{\bm u\in L_{\prim}^\sharp}f(\bm u)^2
+
\sum_{\substack{\bm u,\bm w\in L_{\prim}^\sharp\\
\bm u\ne\bm w}}
f(\bm u)f(\bm w).
$$
The expectation of the first term is the first formula applied to $f^2$.
Hence it is
$$
\frac{1}{2\zeta(n)}
\int_{\re^n}f(\bm x)^2\,d\bm x.
$$
It remains to compute the off--diagonal term.
Two distinct primitive sign classes are linearly independent when only two sign classes are considered.
Indeed, if two primitive lattice vectors are linearly dependent, then they lie in the same rank--one subgroup of $L$.
Primitivity then forces them to differ only by sign.
Thus the off--diagonal term is one fourth of the sum of $f(\bm v)f(\bm w)$ over ordered linearly independent primitive vector pairs.
After choosing representatives $\bm v=\bm qA$ and $\bm w=\bm rA$, the equivalence of real and rational independence shows that these pairs correspond to elements of $\LI_2$.
Since $n\ge3$, we have $2\le n-1$.
Therefore Theorem \ref{thm:macbeath} can be applied with $k=2$.
Using Lemma \ref{lem:mobius-primitive} in both variables and Theorem \ref{thm:macbeath}, we obtain
\begin{eqnarray*}
&&
\ex_{\mu_n}
\sum_{\substack{\bm u,\bm w\in L_{\prim}^\sharp\\
\bm u\ne\bm w}}
f(\bm u)f(\bm w)
\\
&=&
\frac{1}{4}
\sum_{d=1}^{\infty}
\sum_{e=1}^{\infty}
\mu(d)\mu(e)d^{-n}e^{-n} I_f^2
\\
&=&
\frac{1}{4}
\left(\sum_{d=1}^{\infty}\mu(d)d^{-n}\right)^2 I_f^2
\\
&=&
\frac{1}{4\zeta(n)^2}I_f^2.
\end{eqnarray*}
This is equal to $\{\ex_{\mu_n}S_f(L)\}^2$.
Subtracting the square of the mean gives
$$
\Var_{\mu_n}(S_f(L))
=
\frac{1}{2\zeta(n)}
\int_{\re^n}f(\bm x)^2\,d\bm x.
$$
We now pass to a general nonnegative Borel function $f$ such that $f$ and $f^2$ are integrable.
For $j\ge1$, define
$$
f_j(\bm x)
=
\min\{f(\bm x),j\}\chi_{\{\|\bm x\|\le j\}}.
$$
Then $f_j\uparrow f$ and $f_j^2\uparrow f^2$.
Also $S_{f_j}(L)\uparrow S_f(L)$ for every $L$.
The first and second moment formulas therefore pass to the limit by monotone convergence.
The variance formula follows by subtracting the square of the mean.
This proves the proposition.
\end{proof}

\begin{cor}
\label{cor:primitivecounts}
Let $r>0$ and let $M_r(L)$ be the number of primitive sign classes with $\|\bm u\|\le r$.
Let
$$
\mu_r=\frac{v_nr^n}{2\zeta(n)}.
$$
Then
$$
\ex_{\mu_n}M_r(L)=\mu_r,
\qquad
\Var_{\mu_n}(M_r(L))=\mu_r.
$$
Let $L_n$ be distributed according to $\mu_n$ on $\mathbb X_n$.
In particular, if $r_n>0$ and $v_nr_n^n/(2\zeta(n))\to\infty$, then
$$
\frac{M_{r_n}(L_n)}{v_nr_n^n/(2\zeta(n))}\to1
$$
in probability.
For fixed $\rho>1$,
$$
\frac{M_{\rho\GH_n}(L_n)}{\rho^n/(2\zeta(n))}\to1
$$
in probability.
\end{cor}

\begin{proof}
Apply Proposition \ref{prop:weighted} to the indicator function $\chi_{\{\|\cdot\|\le r\}}$.
The convergence follows from Chebyshev's inequality.
If $r_n=\rho\GH_n$, then $v_nr_n^n=\rho^n$.
\end{proof}

\section{Quenched Thermal Concentration for Fixed Factors}
\label{sec:quenched}

We now pass from unweighted primitive counts to weighted Gibbs mass.
For $\bm u\in L_{\prim}^\sharp$, put
$$
Y_n(\bm u)=\frac{1}{n}\log\tau_n(\bm u)
=\log\frac{\|\bm u\|}{\GH_n}.
$$
For a Borel set $A\subset\re$, define the primitive weighted mass
$$
B_{n,c}(A,L)=
\sum_{\bm u\in L_{\prim}^\sharp,\ Y_n(\bm u)\in A}
 g_{n,c}(\tau_n(\bm u)).
$$
We write $B_{n,c}(L)=B_{n,c}(\re,L)$.
The corresponding primitive Gibbs probability of the set $A$ is
$$
\Pi_{n,c}(A,L)=\frac{B_{n,c}(A,L)}{B_{n,c}(L)}.
$$
The exponent governing the first moment is
$$
\Phi_c(y)=y-\frac{c}{2}(e^{2y}-1),
\qquad y\in\re.
$$
The term $y$ is the entropy coming from the growth of the volume variable.
The term $-(c/2)(e^{2y}-1)$ is the Gibbs energy penalty.
The theorem below shows that this elementary variational function gives the actual quenched thermal scale in the primitive ensemble.

\begin{lem}
\label{lem:laplace}
Let $0<c<1$ and put
$$
y_c=\frac{1}{2}\log\frac{1}{c}
$$
and
$$
F(c)=\Phi_c(y_c)=
\frac{1}{2}\log\frac{1}{c}-\frac{1-c}{2}.
$$
Then $F(c)>0$.
Moreover, if
$$
\Psi_c(y)=y-c(e^{2y}-1),
$$
then
$$
\sup_{y\ge0}\Psi_c(y)<2F(c).
$$
\end{lem}

\begin{proof}
The derivative $\Phi_c'(y)=1-ce^{2y}$ vanishes only at $y_c$.
Since $0<c<1$, this point belongs to $(0,\infty)$ and gives the maximum of $\Phi_c$ on $\re$.
This gives the formula for $F(c)$.
The inequality $F(c)>0$ follows from $\log x>1-1/x$ with $x=1/c>1$.

It remains to compare the second--moment exponent.
If $1/2\le c<1$, then $\Psi_c$ is maximized at $0$ and $\sup\Psi_c=0<2F(c)$.
If $0<c<1/2$, then $\Psi_c$ is maximized at $(1/2)\log(1/(2c))$ and
$$
\sup_{y\ge0}\Psi_c(y)=
\frac{1}{2}\log\frac{1}{2c}-\frac{1-2c}{2}.
$$
A direct calculation gives
$$
2F(c)-\sup_{y\ge0}\Psi_c(y)
=
\frac{1}{2}\left(\log\frac{2}{c}-1\right)>0.
$$
This proves the lemma.
\end{proof}

\begin{thm}[High--temperature quenched thermal concentration]
\label{thm:quenched}
Let $L_n$ be distributed according to $\mu_n$ on $\mathbb X_n$.
Assume $0<c<1$.
Then
$$
\frac{1}{n}\log B_{n,c}(L_n)\to F(c)
$$
in probability.
Moreover, for every fixed real number $a$ with $a\ne y_c$,
$$
\Pi_{n,c}((-\infty,a],L_n)\to
\begin{cases}
0, & a<y_c,\\
1, & a>y_c,
\end{cases}
$$
in probability.
\end{thm}

\begin{proof}
For a Borel set $A\subset\re$, define
$$
J_{n,c}(A)=n\int_A\exp\{n\Phi_c(y)\}\,dy
$$
and
$$
K_{n,c}(A)=n\int_A\exp\{n\Psi_c(y)\}\,dy.
$$
The change of variables $s=e^{ny}$ gives
$$
\int_{\re^n}g_{n,c}(\tau_n(\bm x))\chi_A(Y_n(\bm x))\,d\bm x
=
J_{n,c}(A),
$$
where $Y_n(\bm x)=n^{-1}\log(v_n\|\bm x\|^n)$.
The same change of variables gives
$$
\int_{\re^n}g_{n,c}(\tau_n(\bm x))^2\chi_A(Y_n(\bm x))\,d\bm x
=
K_{n,c}(A).
$$
Proposition \ref{prop:weighted} therefore gives
$$
\ex_{\mu_n}B_{n,c}(A,L)=\frac{J_{n,c}(A)}{2\zeta(n)}
$$
and
$$
\Var_{\mu_n}(B_{n,c}(A,L))=\frac{K_{n,c}(A)}{2\zeta(n)}
$$
whenever the two integrals are finite.

We first control the positive volume scale.
For $A=[0,\infty)$, the elementary Laplace estimate gives
$$
\lim_{n\to\infty}\frac{1}{n}\log J_{n,c}([0,\infty))=F(c)
$$
and
$$
\lim_{n\to\infty}\frac{1}{n}\log K_{n,c}([0,\infty))=
\sup_{y\ge0}\Psi_c(y).
$$
By Lemma \ref{lem:laplace}, the second exponent is strictly smaller than $2F(c)$.
Hence
$$
\frac{\Var_{\mu_n}(B_{n,c}([0,\infty),L))}
{\{\ex_{\mu_n}B_{n,c}([0,\infty),L)\}^2}\to0.
$$
It follows that
$$
\frac{B_{n,c}([0,\infty),L_n)}
{\ex_{\mu_n}B_{n,c}([0,\infty),L)}\to1
$$
in probability.
Thus
$$
\frac{1}{n}\log B_{n,c}([0,\infty),L_n)\to F(c)
$$
in probability.

The negative volume scale is negligible on this exponential scale.
Indeed, the Laplace estimate gives
$$
\lim_{n\to\infty}\frac{1}{n}\log J_{n,c}((-\infty,0))=0.
$$
Since $F(c)>0$, Markov's inequality gives
$$
\frac{B_{n,c}((-\infty,0),L_n)}{e^{nF(c)/2}}\to0
$$
in probability.
Therefore $B_{n,c}(L_n)=B_{n,c}(\re,L_n)$ has the same logarithmic limit $F(c)$.

Let $a<y_c$.
Then $\sup_{y\le a}\Phi_c(y)=\Phi_c(a)<F(c)$.
Choose $\delta>0$ so small that $\Phi_c(a)+3\delta<F(c)$.
The Laplace estimate and Markov's inequality imply
$$
\pr\{B_{n,c}((-\infty,a],L_n)>e^{n(\Phi_c(a)+2\delta)}\}\to0.
$$
The preceding lower bound on $B_{n,c}(L_n)$ gives
$$
\pr\{B_{n,c}(L_n)<e^{n(F(c)-\delta)}\}\to0.
$$
These two estimates imply $\Pi_{n,c}((-\infty,a],L_n)\to0$ in probability.

Let $a>y_c$.
Then $\sup_{y\ge a}\Phi_c(y)<F(c)$.
The same argument applied to $[a,\infty)$ gives
$$
\frac{B_{n,c}([a,\infty),L_n)}{B_{n,c}(L_n)}\to0
$$
in probability.
Therefore $\Pi_{n,c}((-\infty,a],L_n)\to1$ in probability.
\end{proof}

\begin{cor}[Fixed--factor primitive visibility curve]
\label{cor:fixedfactor}
Let $L_n$ be distributed according to $\mu_n$ on $\mathbb X_n$.
Let $\gamma>1$ and let $0<c<1$.
Define the primitive Gibbs mass of the $\gamma$--approximate shortest-vector window by
$$
\Pi_{n,c}^{\lambda}(\gamma,L)=
\frac{
\sum_{\bm u\in L_{\prim}^\sharp,\ \|\bm u\|\le\gamma\lambda_1(L)}
 g_{n,c}(\tau_n(\bm u))}
{B_{n,c}(L)}.
$$
Then
$$
\Pi_{n,c}^{\lambda}(\gamma,L_n)\to0
$$
in probability if $c<\gamma^{-2}$, while
$$
\Pi_{n,c}^{\lambda}(\gamma,L_n)\to1
$$
in probability if $c>\gamma^{-2}$ and $c<1$.
This probability is taken with respect to the primitive Gibbs ensemble; the shortest length 
$\lambda_1(L)$ is the usual shortest length of the full lattice, but the sampling measure is primitive.
\end{cor}

\begin{proof}
By Theorem \ref{thm:sodergren}, $\tau_{n,1}(L_n)$ converges in distribution to $T_1$.
Hence $\tau_{n,1}(L_n)^{1/n}\to1$ in probability.
Equivalently, $\lambda_1(L_n)/\GH_n=e^{o(1)}$ in probability.
Let $a=\log\gamma$, so that $\gamma=e^a$ and $a>0$. 
If $c<\gamma^{-2}$, then $a<y_c$. 
Choose $\epsilon>0$ so small that $a+\epsilon<y_c$.
With probability tending to one, 
the window $\|\bm u\|\le\gamma\lambda_1(L_n)$ is contained in the window $Y_n(\bm u)\le a+\epsilon$.
Theorem \ref{thm:quenched} applied to $(-\infty,a+\epsilon]$ then gives the first assertion.

If $c>\gamma^{-2}$ and $c<1$, then $a>y_c$.
Choose $\epsilon>0$ so small that $a-\epsilon>y_c$.
With probability tending to one, the window $Y_n(\bm u)\le a-\epsilon$ is contained in the window $\|\bm u\|\le\gamma\lambda_1(L_n)$.
Theorem \ref{thm:quenched} applied to $(-\infty,a-\epsilon]$ then gives the second assertion.
\end{proof}

\begin{prop}
\label{prop:local-critical-thermal-mass}
Let $L_n$ be distributed according to $\mu_n$ on $\mathbb X_n$.
Assume $0<c<1$ and put
$$
y_c=\frac{1}{2}\log\frac{1}{c}.
$$
For $t\in\re$, define
$$
H(t)=\frac{1}{\sqrt{\pi}}\int_{-\infty}^{t}e^{-x^2}dx.
$$
This is the distribution function of the centered Gaussian law with density
$\pi^{-1/2}e^{-x^2}$, or equivalently of a normal random variable with
variance $1/2$.
Then
$$
\frac{
B_{n,c}((-\infty,y_c+t/\sqrt{n}],L_n)
}{
B_{n,c}(L_n)
}
\longrightarrow H(t)
$$
in probability as $n\to\infty$.
\end{prop}

\begin{proof}
We first consider the interval
$$
A_{n,t}=[0,y_c+t/\sqrt{n}].
$$
For all sufficiently large $n$ this interval is nonempty. Recall that
$$
J_{n,c}(A)=
n\int_A \exp\{n\Phi_c(y)\}dy
$$
and
$$
K_{n,c}(A)=
n\int_A \exp\{n\Psi_c(y)\}dy,
$$
where
$$
\Phi_c(y)=y-\frac{c}{2}(e^{2y}-1)
$$
and
$$
\Psi_c(y)=y-c(e^{2y}-1).
$$
By Proposition \ref{prop:weighted},
$$
\ex_{\mu_n}B_{n,c}(A,L)=\frac{J_{n,c}(A)}{2\zeta(n)}
$$
and
$$
\Var_{\mu_n}(B_{n,c}(A,L))
=
\frac{K_{n,c}(A)}{2\zeta(n)}.
$$

We compute the first moment near the maximum of $\Phi_c$.
The function $\Phi_c$ has its unique maximum on $[0,\infty)$ at $y_c$,
and
$$
\Phi_c(y_c)=F(c),\qquad \Phi_c''(y_c)=-2.
$$
Choose $\delta>0$ so small that $y_c-\delta>0$.
Since $y_c$ is the unique maximum of $\Phi_c$ on $[0,\infty)$, there is
$\eta>0$ such that
$$
\Phi_c(y)\le F(c)-\eta
$$
for all $y\in[0,\infty)$ with $|y-y_c|\ge\delta$.
Thus the contribution of this region to $J_{n,c}([0,\infty))$ and to
$J_{n,c}(A_{n,t})$ is exponentially smaller than $e^{nF(c)}$.

On the interval $|y-y_c|<\delta$, put
$$
y=y_c+\frac{x}{\sqrt{n}}.
$$
Then
$$
J_{n,c}(A_{n,t})
=
\sqrt{n}e^{nF(c)}
\int_{-\sqrt{n}y_c}^{t}
\exp\{n(\Phi_c(y_c+x/\sqrt{n})-F(c))\}dx
$$
up to an exponentially smaller error from the part outside
$(y_c-\delta,y_c+\delta)$.
Taylor's formula gives, locally uniformly in $x$,
$$
n(\Phi_c(y_c+x/\sqrt{n})-F(c))\longrightarrow -x^2 .
$$
Moreover, after decreasing $\delta$ if necessary, there is a constant
$\kappa>0$ such that
$$
\Phi_c(y_c+z)\le F(c)-\kappa z^2
$$
for $|z|\le\delta$.
Hence, in the variable $x=\sqrt n z$, the local integrand is bounded by
$e^{-\kappa x^2}$.
Therefore dominated convergence, together with the exponentially small
contribution outside the fixed neighborhood of $y_c$, gives
$$
J_{n,c}(A_{n,t})
=
\sqrt{n}e^{nF(c)}
\left\{
\int_{-\infty}^{t}e^{-x^2}dx+o(1)
\right\}.
$$
The same argument with the moving endpoint removed gives
$$
J_{n,c}([0,\infty))
=
\sqrt{n}e^{nF(c)}
\left\{
\int_{-\infty}^{\infty}e^{-x^2}dx+o(1)
\right\}
=
\sqrt{\pi n}e^{nF(c)}{1+o(1)}.
$$
Consequently,
$$
\frac{J_{n,c}(A_{n,t})}{J_{n,c}([0,\infty))}
\longrightarrow
\frac{1}{\sqrt{\pi}}\int_{-\infty}^{t}e^{-x^2}dx
=H(t).
$$
We next show that the corresponding random weighted sums are concentrated
around their means. By the elementary Laplace upper bound,
$$
\limsup_{n\to\infty}
\frac{1}{n}\log K_{n,c}([0,\infty))
\le
\sup_{y\ge0}\Psi_c(y).
$$
By Lemma \ref{lem:laplace}, the right hand side is strictly smaller than
$2F(c)$. Choose $G$ such that
$$
\sup_{y\ge0}\Psi_c(y)<G<2F(c).
$$
Then, for all sufficiently large $n$,
$$
K_{n,c}([0,\infty))\le e^{nG}.
$$
Since $J_{n,c}(A_{n,t})$ is of order $\sqrt n e^{nF(c)}$, Chebyshev's
inequality gives
$$
\frac{
B_{n,c}(A_{n,t},L_n)
}{
\ex_{\mu_n}B_{n,c}(A_{n,t},L)
}
\longrightarrow 1
$$
in probability. The same argument gives
$$
\frac{
B_{n,c}([0,\infty),L_n)
}{
\ex_{\mu_n}B_{n,c}([0,\infty),L)
}
\longrightarrow 1
$$
in probability.

It remains to remove the negative part of the $Y_n$--scale. For $y<0$,
the function $\Phi_c(y)$ has supremum $0$, whereas $F(c)>0$. Hence
$$
\limsup_{n\to\infty}
\frac{1}{n}\log J_{n,c}((-\infty,0))\le 0.
$$
By Markov's inequality,
$$
\frac{
B_{n,c}((-\infty,0),L_n)
}{
J_{n,c}([0,\infty))
}
\longrightarrow 0
$$
in probability. On the other hand, the preceding concentration estimate
implies that $B_{n,c}([0,\infty),L_n)$ is of order
$J_{n,c}([0,\infty))$ in probability. Therefore
$$
\frac{
B_{n,c}((-\infty,0),L_n)
}{
B_{n,c}([0,\infty),L_n)
}
\longrightarrow 0
$$
in probability. Thus the negative part of the $Y_n$--scale does not affect
the limiting ratio. Combining this fact with the Laplace ratio above gives
$$
\frac{
B_{n,c}((-\infty,y_c+t/\sqrt{n}],L_n)
}{
B_{n,c}(L_n)
}
\longrightarrow H(t)
$$
in probability. This proves the proposition.
\end{proof}

\begin{cor}[Critical fixed-factor primitive visibility]
\label{cor:critical-fixed-factor-visibility}
Let $L_n$ be distributed according to $\mu_n$ on $\mathbb X_n$.
Let $\gamma>1$ and put
$$
c=\gamma^{-2}.
$$
Then
$$
\Pi_{n,c}^{\lambda}(\gamma,L_n)
\longrightarrow
\frac{1}{2}
$$
in probability as $n\to\infty$.
\end{cor}

\begin{proof}
Let $a=\log\gamma$.
Since $c=\gamma^{-2}$, we have
$$
a=\frac{1}{2}\log\frac{1}{c}=y_c.
$$
The condition
$$
\|\bm u\|\le \gamma\lambda_1(L_n)
$$
is equivalent to
$$
Y_n(\bm u)\le a+R_n,
$$
where
$$
R_n=\frac{1}{n}\log \tau_{n,1}(L_n).
$$
By Theorem \ref{thm:sodergren}, $\tau_{n,1}(L_n)$ converges in distribution
to the first point $T_1$ of the limiting Poisson process. In particular,
$$
\log \tau_{n,1}(L_n)=O_{\mathbb P}(1).
$$
Hence
$$
\sqrt n R_n
=
\frac{\log \tau_{n,1}(L_n)}{\sqrt n}
\longrightarrow 0
$$
in probability.

Set
$$
X_n=\Pi_{n,c}^{\lambda}(\gamma,L_n).
$$
Fix $\epsilon>0$ and define
$$
X_n^-=
\frac{
B_{n,c}((-\infty,y_c-\epsilon/\sqrt n],L_n)
}{
B_{n,c}(L_n)
}
$$
and
$$
X_n^+=
\frac{
B_{n,c}((-\infty,y_c+\epsilon/\sqrt n],L_n)
}{
B_{n,c}(L_n)
}.
$$
Since $\sqrt n R_n\to0$ in probability, the event
$$
-\frac{\epsilon}{\sqrt n}\le R_n\le \frac{\epsilon}{\sqrt n}
$$
has probability tending to one. On this event,
$$
X_n^-\le X_n\le X_n^+.
$$
By Proposition \ref{prop:local-critical-thermal-mass},
$$
X_n^-\longrightarrow H(-\epsilon),
\qquad
X_n^+\longrightarrow H(\epsilon)
$$
in probability. Hence, for every $\delta>0$,
$$
\limsup_{n\to\infty}
\pr\{X_n<H(-\epsilon)-\delta\}=0
$$
and
$$
\limsup_{n\to\infty}
\pr\{X_n>H(\epsilon)+\delta\}=0.
$$
Letting $\epsilon\downarrow0$ and using the continuity of $H$ together
with $H(0)=1/2$, we obtain
$$
X_n\longrightarrow\frac{1}{2}
$$
in probability. This is the desired assertion.
\end{proof}

\section{Discussion and Conclusion}
\label{sec:discussion}

The main result is that random-lattice Gibbs ensembles have two related visibility structures.
Near the exact edge, the threshold is $c=1$.
Below and at this threshold, every fixed $e^{a/n}$--approximation window is invisible.
Above it, the Gibbs weights on the edge converge to a Poisson--Dirichlet mass partition.
This is an edge condensation phenomenon for a random geometric Gibbs ensemble.

For fixed approximation factors $\gamma>1$, the relevant scale is different.
The exponential coordinate $Y_n=n^{-1}\log\tau_n$ has entropy term $y$ and energy term $-(c/2)(e^{2y}-1)$.
The resulting function $\Phi_c$ is maximized at $y_c=(1/2)\log(1/c)$ when $0<c<1$.
The primitive weighted moment formula turns this saddle--point picture into a quenched concentration theorem.
This gives the primitive fixed-factor visibility curve $c=\gamma^{-2}$ in probability.
More precisely, below this curve the primitive Gibbs mass of the fixed approximation window vanishes, above this curve it tends to one, and on the critical curve it tends to $1/2$.
In this sense the curve is not only a formal saddle point.
It describes where the typical primitive Gibbs mass is located in the high-dimensional random environment.

The present results may also be read as a thermodynamic reference point for idealized Gibbs targets.
If an ideal sampler produced the Gibbs distribution, then the theorems above describe whether an approximate shortest-vector window has visible mass.
If one could prepare a coherent version of the same Gibbs distribution, amplitude amplification \cite{Grover96,Brassard02} would change a success probability $p$ into a cost of order $p^{-1/2}$, but it would not remove the visibility thresholds found here.
We do not construct a Gibbs sampler or a quantum algorithm, and no claim on the complexity of approximate SVP is made.

The fixed-factor theorem is deliberately narrower than the edge theorem.
It concerns the primitive Gibbs ensemble and assumes $0<c<1$.
Extending the fixed-factor visibility result to the full Gibbs ensemble remains a natural further question.
Other open problems include concrete Markov chains or quantum Gibbs samplers \cite{Temme11}, their mixing or preparation cost, and comparisons with structured lattice families such as $q$--ary, NTRU, and module or ring LWE lattices.
The primitive concentration theorem gives a clean starting point for such future comparisons.

\section*{Data Availability Statement}
No new data were created or analyzed in this study.

\section*{Conflict of interest}
The author declares no conflicts of interest.

\section*{Ethical statement}
Ethical approval was not required for this theoretical study, which does not involve human participants, human data, or animals.

\section*{Acknowledgments}
The author thanks Prof. Takuya Mine for carefully reading an earlier version of the manuscript, 
especially the arguments in Section~\ref{sec:primitive}, and for helpful comments.
This work was supported by JSPS KAKENHI Grant Number 25K07752.


\small
\begin{thebibliography}{99}

\bibitem{Derrida80}
B. Derrida,
Random--energy model: Limit of a family of disordered models,
Physical Review Letters 45 (1980), no. 2, 79--82.
\url{https://doi.org/10.1103/PhysRevLett.45.79}

\bibitem{Derrida81}
B. Derrida,
Random--energy model: An exactly solvable model of disordered systems,
Physical Review B 24 (1981), no. 5, 2613--2626.
\url{https://doi.org/10.1103/PhysRevB.24.2613}

\bibitem{Sodergren11}
A. S\"odergren,
On the Poisson distribution of lengths of lattice vectors in a random lattice,
Mathematische Zeitschrift 269 (2011), no. 3--4, 945--954.
\url{https://doi.org/10.1007/s00209-010-0772-8}

\bibitem{Sodergren10}
A. S\"odergren,
On the value distribution and moments of the Epstein zeta function to the right of the critical strip,
Journal of Number Theory 131 (2011), no. 7, 1176--1208.
\url{https://doi.org/10.1016/j.jnt.2010.12.003}

\bibitem{Sodergren13}
A. S\"odergren,
On the value distribution of the Epstein zeta function in the critical strip,
Duke Mathematical Journal 162 (2013), no. 1, 1--48.
\url{https://doi.org/10.1215/00127094-1903389}

\bibitem{Ajtai96}
M. Ajtai,
Generating hard instances of lattice problems,
extended abstract,
in Proceedings of the 28th ACM Symposium on Theory of Computing,
ACM, 1996, pp. 99--108.
\url{https://doi.org/10.1145/237814.237838}

\bibitem{MG02}
D. Micciancio and S. Goldwasser,
{\em Complexity of Lattice Problems: A Cryptographic Perspective},
Springer, 2002.
\url{https://doi.org/10.1007/978-1-4615-0897-7}

\bibitem{Regev09}
O. Regev,
On lattices, learning with errors, random linear codes, and cryptography,
Journal of the ACM 56 (2009), no. 6, Article 34, 40 pp.
\url{https://doi.org/10.1145/1568318.1568324}

\bibitem{MicciancioRegev07}
D. Micciancio and O. Regev,
Worst--case to average--case reductions based on Gaussian measures,
SIAM Journal on Computing 37 (2007), no. 1, 267--302.
\url{https://doi.org/10.1137/S0097539705447360}

\bibitem{Siegel45}
C. L. Siegel,
A mean value theorem in geometry of numbers,
Annals of Mathematics 46 (1945), no. 2, 340--347.
\url{https://doi.org/10.2307/1969027}

\bibitem{Kingman75}
J. F. C. Kingman,
Random discrete distributions,
Journal of the Royal Statistical Society. Series B (Methodological)
37 (1975), no. 1, 1--15.
\url{https://doi.org/10.1111/j.2517-6161.1975.tb01024.x}

\bibitem{PitmanYor97}
J. Pitman and M. Yor,
The two--parameter Poisson--Dirichlet distribution derived from a stable subordinator,
Annals of Probability 25 (1997), no. 2, 855--900.
\url{https://doi.org/10.1214/aop/1024404422}

\bibitem{Rogers55Mean}
C. A. Rogers,
Mean values over the space of lattices,
Acta Mathematica 94 (1955), 249--287.
\url{https://doi.org/10.1007/BF02392493}

\bibitem{Rogers55Moments}
C. A. Rogers,
The moments of the number of points of a lattice in a bounded set,
Philosophical Transactions of the Royal Society of London. Series A
248 (1955), no. 945, 225--251.
\url{https://doi.org/10.1098/rsta.1955.0015}

\bibitem{MacbeathRogersI}
A. M. Macbeath and C. A. Rogers,
A modified form of Siegel's mean--value theorem,
Proceedings of the Cambridge Philosophical Society
51 (1955), no. 4, 565--576.
\url{https://doi.org/10.1017/S0305004100030656}

\bibitem{MacbeathRogersII}
A. M. Macbeath and C. A. Rogers,
A modified form of Siegel's mean--value theorem. II,
Proceedings of the Cambridge Philosophical Society
54 (1958), no. 3, 322--326.
\url{https://doi.org/10.1017/S030500410003351X}

\bibitem{Grover96}
L. K. Grover,
A fast quantum mechanical algorithm for database search,
in Proceedings of the 28th ACM Symposium on Theory of Computing,
ACM, 1996, pp. 212--219.
\url{https://doi.org/10.1145/237814.237866}

\bibitem{Brassard02}
G. Brassard, P. H{\o}yer, M. Mosca, and A. Tapp,
Quantum amplitude amplification and estimation,
in {\em Quantum Computation and Information},
Contemporary Mathematics 305,
American Mathematical Society, 2002, pp. 53--74.
\url{https://doi.org/10.1090/conm/305/05215}

\bibitem{Temme11}
K. Temme, T. J. Osborne, K. G. H. Vollbrecht, D. Poulin, and F. Verstraete,
Quantum Metropolis sampling,
Nature 471 (2011), no. 7336, 87--90.
\url{https://doi.org/10.1038/nature09770}

\end{thebibliography}
\end{document}